\documentclass{amsart}

\usepackage{amsmath}%
\usepackage{amsfonts}%
\usepackage{amssymb}%
\usepackage{graphicx}
\usepackage{color}
\usepackage[left=3cm,top=2cm,right=3cm,bottom=2cm,includehead=true,includefoot=true]{geometry}
\usepackage{cancel}
\usepackage{rotating}
\usepackage{subfigure}
\usepackage{hyperref}
\usepackage[normalem]{ulem}
\usepackage[utf8x]{inputenc}
\newcommand{\authorfootnotes}{\renewcommand\thefootnote{\@fnsymbol\c@footnote}}%
\makeatletter

\def\'#1{{\if #1i{\accent"13\i}\else {\accent"13 #1}\fi}}

\def\R{\mathbb{R}}

\def\e{\varepsilon}

\newtheorem{theo}{Theorem}[section]
\newtheorem{lemme}[theo]{Lemma}

\newtheorem{defi}[theo]{Definition}

\newtheorem{prop}[theo]{Proposition}

\newtheorem{remark}[theo]{Remark}

\numberwithin{equation}{section}

\definecolor{orange}{rgb}{1.00,0.50,0.0}
\definecolor{violet}{rgb}{0.46,0.39,0.87}
\definecolor{green}{rgb}{0.2,0.4,0.6}

\DeclareMathAlphabet{\mathonebb}{U}{bbold}{m}{n}

\begin{document}

\title[Homeoprotein Diffusion and Boundary Stabilization]{Competition and boundary formation in heterogeneous media: Application to neuronal differentiation}
\maketitle

\author{
\begin{center}
 Benoit Perthame\footnote{UPMC Univ Paris 06 and CNRS UMR 7598, Laboratoire Jacques-Louis Lions, F-75005, Paris, France. Email B. P.:
benoit.perthame@ljll.math.upmc.fr}\qquad\qquad Cristobal Qui\~ninao$^{1,}$\footnote{Mathematical Neuroscience Team, Center for Interdisciplinary Research in Biology (CIRB), F-75005 Paris, France. Email C.Q.: cristobal.quininao@college-de-france.fr, Email J.T.: jonathan.touboul@college-de-france.fr}\qquad\qquad Jonathan Touboul$^{2,}$\footnote{INRIA Mycenae Team, Paris-Rocquencourt Center}
\end{center}
}

\begin{abstract}
We analyze an inhomogeneous system of coupled reaction-diffusion equations representing the dynamics of gene expression during differentiation of nerve cells. The outcome of this developmental phase is the formation of distinct functional areas separated by sharp and smooth boundaries. It proceeds through the competition between the expression of two genes whose expression is driven by monotonic gradients of chemicals, and the products of gene expression undergo local diffusion and drive gene expression in neighboring cells. The problem therefore falls in a more general setting of species in competition within a non-homogeneous medium. We show that in the limit of arbitrarily small diffusion, there exists a unique monotonic stationary solution, which splits the neural tissue into two winner-take-all parts at a precise boundary point: on both sides of the boundary, different neuronal types are present. In order to further characterize the location of this boundary, we use a blow-up of the system and define a traveling wave problem parametrized by the position within the monotonic gradient: the precise boundary location is given by the unique point in space at which the speed of the wave vanishes. 
\end{abstract}




\section{Introduction}
In this paper we undertake a rigorous mathematical analysis of the boundary formation in a model of developing tissue. Our motivation can be traced back to the work of Alan Turing in the middle of last century, that lead to his celebrated theory of instabilities~\cite{turing:52}. In his paper, Turing proposed, before substantial knowledge about the development and maturation of living systems was acquired, that the determination of territories was the result of the competition between different chemical substances, he called morphogens, that were reacting together and diffusing, in the presence of a third specie which acts as a catalyst on the expression of both species. In a certain regime of diffusion, these equations yield what we now call \emph{Turing patterns}, that define a partition of the tissue into differentiated areas (expressing one or the other chemical specie), whose exact shape and location are unpredictable and depend on the initial condition. 

In contrast to this indeterminacy of the boundary location in Turing's model, morphogenesis in living systems is an extremely reliable process. Actually, precision of the boundary location is crucial from an evolutionary perspective, in that it ensures proper transmission of essential hereditary patterns. Notwithstanding this qualitative distinction, several years after introduction of Turing's model, biological experiments validated Turing's intuition: transcription factors (called homeoproteins) expressed in cells during development have been shown to have self-activating and reciprocal inhibitor properties as in Turing's theory, but moreover, where shown to have the property to exit the cellular nucleus and membrane and enter the neighboring cells nucleus where it exerts its transcriptional properties~\cite{prochiantz-joliot:03,layalle-etal:11}. However, in contrast to the initial Turing  model, the catalyst chemical specie show a specific spatial organization: it forms one-dimensional monotonic gradients of concentration~\cite{brunet-etal:07}. This arrangement of catalysts along gradients lead to the development of the french flag model (FFM)~\cite{wolpert:69}. This model assumes no diffusion of genetic material, but only all-or-none response to specific thresholds of the catalyst gradient, therefore yielding boundary at a specific location in space. However, this model is relatively sensitive to noise and necessitates to introduce finer mechanisms in order to ensure robustness and accuracy of the boundary location~\cite{meinhardt1978space,gierer:72}.

Combining both phenomena of non cell-autonomous activity (small diffusion of transcription factors, acting as Turing morphogens) and graded expression of a catalyst (FFM-like model) lead to a recently developed minimalistic model of boundary formation~\cite{quininao-prochiantz} reproducing in a parsimonious way both reliability and accuracy of boundary location. This model is given by nonlinear parabolic equations with spatially-dependent coefficients. Simulations indeed showed that in the absence of diffusion, there is no clear separation in two regions, but even very small diffusions disambiguate the differentiation process and lead to a clear definition of the boundary. The object of the paper is to rigorously understand this stabilization in the regime of small diffusions. The mathematical problem we shall be analyzing is actually much more general than the problem of neurodevelopment that motivates the study. Indeed, systems characterized by the competition of two species that are self-activating and reciprocal inhibitor are ubiquitous in life science, and extend to spatially extended population models, large-scale systems of bacterias and social interactions. The particularity of the model we shall analyze, and which may occur in different situations in the cited domains, is precisely the presence of the non-spatially homogeneous catalyst, producing predictable and reproducible patterns. 

Due to the ubiquity of such competing systems in life science, we shall propose here a general model supporting reliable pattern formation, and relevant to biology. To this purpose, we complete this introduction by briefly exposing details on neuronal differentiation, before introducing the model we shall be investigating and summarizing our main mathematical results. 

\subsection{Biological motivation}\label{motivation}

Let us make more precise the model we have in mind in our developments. The central question we shall address the emergence of reliable boundaries in the developing nervous system. The neural tube indeed develops into a complex functional and anatomical architecture endowed with complex connectivity patterns~\cite{raff}. The size and shape of functional areas in the cortex is of primary importance: it conditions acquisition of functions, and disruptions are associated to severe conditions, including neuropsychiatric and cognitive disorders~\cite{uhlhaas2011development,garey2010cortical}. In the beginning of this century, biologists analyzed developmental genes transcription factors, and showed that these are endowed with non cell-autonomous activity (they belong the \emph{homeoprotein} family), thanks to two short peptidic sequences present in their DNA-binding domain~\cite{joliot-prochiantz:04}.  These transcription factors have the capability to exit the nucleus of the cells, leave the intracellular medium and penetrate the nucleus of neighboring cells where they exert they transcriptional activity. This direct signaling was experimentally demonstrated in vivo during development in the zebrafish~\cite{lesaffre-etal:07,wizenmann-etal:09},  or involved in plasticity of adult networks~\cite{beurdeley-etal:12,sugiyama-etal:08,spatazza-etal:13,brunet-etal:07}. The spatial extension and rate of this process are very low: transcription factors can diffuse and reach at most three cell ranks~\cite{layalle2011engrailed}, and since the diffusion is passive, important loss reduce the effective number of transcription factors involved. Notwithstanding, it was shown recently~\cite{quininao-prochiantz} in an elementary model of neurodevelopment that even very low diffusion had major effects on the outcome of the differentiation process. Indeed, in the absence of diffusion, there is an ambiguity in the differentiation in a specific region of the neural tissue, which yield imprecise boundaries that are not reproducible, and sensitively depend on initial condition and possible heterogeneity or noise, but in the presence of small diffusion, the location of the boundary is highly reliable, and the differentiation yields a smooth boundary. 

Understanding this dramatic regularization is precisely the object of the present paper. This problem falls in the frame of the competition of two diffusing species $A$ and $B$ that are reciprocal inhibitor and self-activating, with saturation and spatially heterogeneous production rates $H_A(x,A,B)$ and $H_B(x,A,B)$ (depending on the cell location $x$). In the neurodevelopment problem, transcription factors expressed by two genes $G_A$ and $G_B$ constitute our two competing species, and the space heterogeneity corresponds to the graded concentration of morphogens. For simplicity, we shall restrict here our analysis to a one-dimensional case\footnote{Generalization to higher dimensions in situations where geometry of the space and the spatial variations along gradients are sufficiently simple can be handled in the same manner. In~\cite{quininao-prochiantz}, we propose a two-dimensional extension of this property. } in which the differentiating tissue is along the interval $[0,1]$. A schematic version of the model is plotted is Figure~\ref{fig:Model}.
\begin{figure}[h]
	\centering
		\includegraphics[width=.7\textwidth]{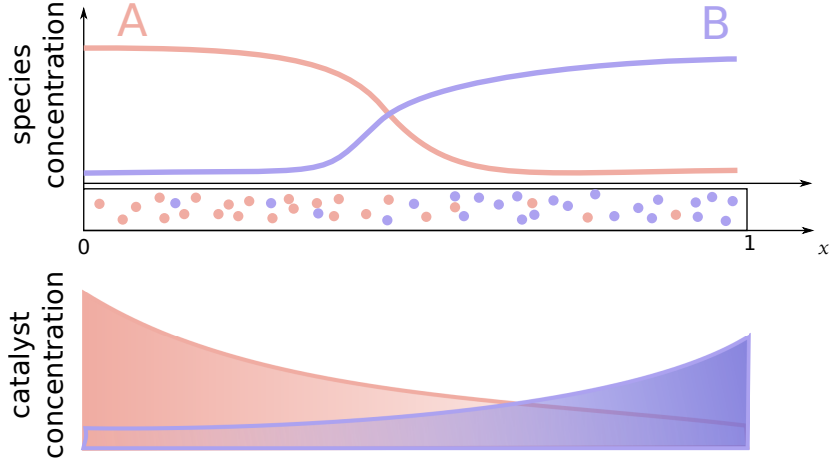}
	\caption{Our model describes the dynamics of two species in competition ($A$, pink and $B$, violet) responding to monotonic resource distributions (bottom line), with reciprocally inhibitory activity and subject to local diffusion.}
	\label{fig:Model}
\end{figure}

\subsection{General model and main result} 
The above description naturally leads to the definition of the following system of reaction-diffusion equations: 
\begin{equation}\label{eq:GeneralElliptic}
	\begin{cases}
	-\varepsilon d_A\Delta A = A \, H_A(x,A,B),  \qquad 0 < x < 1, 	 \\[5pt]
	-\varepsilon d_B\Delta B = B\, H_B(x,A,B),
	\end{cases}
\end{equation}
with Robin type boundary conditions stated below.

Here $H_A$ and $H_B$ are maps from $[0,1]\times \R_+\times \R_+$ on $\R$, assumed to be of class $C^2$. Based on our description of the phenomena, we assume that, for $0 < x < 1$, $A>0,\;B>0$, 
\begin{equation}\label{GH2}
	\begin{cases}
 H_A(x,0,0)>0, \qquad \qquad H_B(x,0,0)>0,  
 \\[5pt]
 \partial_x H_A(x,A,B)<0,\qquad \partial_x H_B(x,A,B)>0, 
 \\[5pt]
\partial_BH_A(x,A,B)<0,\qquad \partial_AH_B(x,A,B)<0,
\end{cases}
\end{equation}
which can be interpreted as follows: on the one hand, the  morphogen gradients do not vanish and vary monotically, on the other hand the system expresses competition between species $A$ and $B$.

Because, we are interested in the limit $\e=0$, the solutions of \eqref{eq:GeneralElliptic} in the absence of diffusion are useful.  We assume that there exists two  solutions $(F_A(x)>0,0)$ and $(0,F_B(x)>0)$ 
\begin{equation}\label{GH1a}
H_A \big(x, F_A(x), 0 \big) =0, \qquad H_B \big(x, 0, F_B(x) \big) =0,
\end{equation}
and that they are respectively  stable for  $x\in (0,x_a)$ and for $x\in (x_b,1)$, with $x_a>x_b$, i.e., there exists a bistable zone. It means that the linearized matrix at $(F_A(x),0)$ have negative eigenvalues for $x\in (0,x_a)$. The same holds at $(0,F_B(x))$ for $x\in (x_b,1)$. Moreover, we assume
\begin{equation}\label{GH1b}
H_B(x,F_A(x),0)>0 \quad \text{for }   x>x_a >x_b, \qquad H_A(x,0,F_B(x)) > 0 \quad \text{ for }  x<x_b<x_a.
\end{equation}
The first inequality, for instance, can be interpreted as follows: for $x>x_a$, $A$ loses stability because resource concentration of $B$ overcomes inhibition from $A$.

Finally, we assume that there exists a unique additional solution $(A^*(x)>0,B^*(x)>0)$ in the interval $(x_b,x_a)$ which is a saddle, i.e.
\begin{equation}\label{GH1c}  \begin{cases}
H_A(x,A^*(x),B^*(x))=0,  \qquad \quad \text{ for } x_b < x < x_a, 
\\[5pt]
 \partial_AH_A(x,A^*,B^*) \partial_BH_B(x,A^*,B^*)- \partial_BH_A(x,A^*,B^*) \partial_AH_B(x,A^*,B^*)  < 0 ,
 \\[5pt]
 \partial_AH_A(x,A^*,B^*)<0,\qquad \partial_BH_B(x,A^*,B^*)<0
 \end{cases}
 \end{equation}
that simply express the negativity of the determinant of the Jacobian matrix at this point:
	\begin{equation}\nonumber
		\left \vert
	 \begin{array}{ll}
	  A^*\;\partial_AH_A(x,A^*,B^*) & A^*\;\partial_BH_A(x,A^*,B^*) \\
	  B^*\;\partial_AH_B(x,A^*,B^*) & B^*\;\partial_BH_B(x,A^*,B^*)
	 \end{array}
	\right \vert < 0 .
	\end{equation}

In order to complete the definition of our system~\eqref{eq:GeneralElliptic}, we need to specify the boundary conditions considered. We are interested in solutions in which the system decomposes the domain into two separate areas in which $A$ or $B$ dominate. In the limit where $\varepsilon$ going to zero, it is therefore natural to consider that the system is subject to Dirichlet boundary conditions, but as the diffusion coefficient increases, loss of transcription factor through the boundary becomes increasingly prominent. These mechanisms correspond to Robin (also called third type) boundary conditions:
\begin{equation}\label{eq:BoundaryConditions}
\begin{cases}
A(0)-\sqrt{\varepsilon}\,\frac{\partial}{\partial x}A(0)=F_A(0),\qquad A(1)+\sqrt{\varepsilon}\,\frac{\partial}{\partial x}A(1)=0,
\\[5pt]
B(0)-\sqrt{\varepsilon}\,\frac{\partial}{\partial x}B(0)=0,\qquad \quad \ B(1)+\sqrt{\varepsilon}\,\frac{\partial}{\partial x}B(1)=F_B(1).
\end{cases}
\end{equation} 
\\

At this level of generality, assumptions~\eqref{GH2}--\eqref{GH1c} may appear formal. These are actually very natural, and we refer to Section~\ref{sec:model} for a basic example where they are satisfied. They formulate in a general fashion the elements of our problem: the first assumption expresses the existence of two stable differentiated states at both ends of the differentiating tissue in the absence of diffusion, whose domain of stability may overlap. In other words, in the absence of diffusion, levels of concentration of morphogen are sufficient to support differentiated states at the boundaries of the interval, and there exists generically an overlap between these two regions. Within this overlap (in the bistable regime), a saddle fixed point naturally emerges between the two solutions due to the properties of planar vector fields, and in our system, at this fixed point, concentrations of $A$ and $B$ perfectly balance the concentrations of morphogen. 

The main result that we will be demonstrating in the present manuscript is the fact that in the presence of small diffusion, a clear boundary between two differentiated domains exists and is unique, and may be characterized univocally. In detail, we shall demonstrate the following:
\begin{theo}\label{theo1}
 Under assumptions~\ref{GH2}--\ref{GH1c}, there exists a classical stationary solution $(A_\e, B_\e)$ of~\eqref{eq:GeneralElliptic} which satisfies 
\begin{equation}\label{eq:monotone}
\frac{d}{d x}A_\e(x) <0, \qquad \frac{d}{d x}B_\e(x) >0,
\end{equation}
and is obtained as $t\to \infty$ in the corresponding parabolic equation. Moreover
\begin{itemize}
 \item[(i)] As $ \varepsilon \to 0$,  $(A_\e,B_\e)$ converges a.e. towards a pair $(A_0,B_0)$. These maps are discontinuous at   some point $x^*\in [x_b,x_a]$ and have disjoint supports 
 $$
 {\rm supp}(A_0)=[0,x^*] \qquad \text{and} \qquad\ {\rm supp}(B_0)= [x^*,1].
 $$
 \item[(ii)] The point $x^*$ is characterized by the relation $c(x^*)=0$ where $c(\cdot)$ represents the speed of propagation of a traveling wave problem parametrized by $x$ (see equation~\eqref{eq:TravWave}).
 \end{itemize}
\end{theo}

This qualitative result falls in the class of free boundary problems, a well developed asymptotic theory in the frame of homogeneous elliptic or parabolic semilinear equations and systems~\cite{CarrPego, WF.PS:86, CE.PS:89, GB.LE.PS:90, BGS97, PS:98}. 
As reviewed in~\cite{C.I.L:92}, these results generally rely on the definition and analysis of viscosity solutions of the resulting Hamilton-Jacobi equation. The second point of the theorem involves a traveling wave with inhomogeneous speed. A vast literature, in particular in the domain of mathematical ecology of competing populations, have been interested in related questions (see e.g.~\cite{VVV, Xin_rev}). Most applications of this theory are related to front propagations and rules to compute their speeds and invasion properties in homogeneous or heterogeneous environments \cite{B_H_02}. Here, we ask a distinct question concerned with the determination of the precise point where a transition between two stable states occurs.

Theorem~\ref{theo1} will therefore show the existence of monotonic solutions. The monotonicity is a consequence of analogous properties of the equilibria in the absence of diffusion, which can be readily proved under the current assumptions. Similarly, the monotonicity of equilibria $A^*(x)$ and $B^*(x)$ can be characterized. This is the object of the following:
\begin{lemme}\label{lem:Monotonicity}
	Under assumption~\eqref{GH2}, the functions defined in \eqref{GH1a} and \eqref{GH1c} satisfy
	\begin{equation}\label{MonotnicityAB}
	\frac{d}{dx} F_A(x) < 0  \quad\text{for }\;  x \in [0, x_a), \qquad  \frac{d}{dx}F_B(x) > 0 \quad \text{for }\;  \in (x_b, 1], 
		 \end{equation}
	\begin{equation}\label{inversed}
	   \frac{d}{dx}A^*(x)>0\quad\text{and}\quad \frac{d}{dx}B^*(x)<0,\qquad x_b<x<x_a.
	 \end{equation}
\end{lemme}
\begin{proof}
	 Since $F_A(x)$ is a fixed point of the system in the absence of diffusion, we have:
	\begin{equation}\nonumber
	 \frac{d}{dx}H_A(x,F_A(x),0)= \partial_xH_A(x,F_A(x),0)+\partial_AH_A(x,F_A(x),0)\frac{d}{dx}F_A(x)=0,
	\end{equation}
	 and therefore
	\begin{equation}\nonumber
	 \frac{d}{dx}F_A(x)=-\frac{\partial_xH_A(x,F_A(x),0)}{\partial_AH_A(x,F_A(x),0)}.
	\end{equation}
	Assumption~\eqref{GH1b} ensures that $ \partial_A H_A(x,F_A(x),0)<0$ readily implies that $\frac{d}{dx}F_A(x)<0$ for $x$ in $[0,x_a)$. By a similar argument, $\frac{d}{dx}F_B(x)>0$ for any $x$ in $(x_b,1]$.

	Hypotheses~\eqref{GH2} and~\eqref{GH1c}  also constrain the monotonicity of $A^*$ and $B^*$. Indeed, since the vector function $(H_A,H_B)$ is constant along the curve $(x,A^*(x),B^*(x))$, we have
	\begin{equation}\nonumber
	 \frac{d}{dx}A^*(x)=\frac{\partial_B H_A\;\partial_xH_B-\partial_xH_A\;\partial_BH_B}{\partial_AH_A\;\partial_BH_B-\partial_BH_A\;\partial_AH_B},\qquad\frac{d}{dx}B^*(x)=\frac{\partial_xH_A\;\partial_AH_B-\partial_AH_A\;\partial_xH_B}{\partial_AH_A\;\partial_BH_B-\partial_BH_A\;\partial_AH_B}.
	\end{equation}
	Using the assumptions \eqref{GH2} and \eqref{GH1c}, we conclude the inequalities \eqref{inversed}.
\end{proof}

The manuscript is devoted to the demonstration of Theorem~\ref{theo1}, and to the development of an application to a specific model of neuronal differentiation. We shall start by proving the existence of a monotonic solution of the elliptic system~\eqref{eq:GeneralElliptic}, \eqref{eq:BoundaryConditions} by analyzing the long-time properties of the associated parabolic system. The proof of the existence of monotonic solutions and the characterization of the boundary combines stability and monotonicity arguments, WKB asymptotics and a suitable dilation of the spatial variable. The proof proceeds as follows: the limit where $\e \to 0$ is investigated in section~\ref{sec:vanishes} and we will show existence and uniqueness of the boundary point $x^*$ for small diffusions, and in section~\ref{sec:characterization}, we characterize the boundary point $x^*$ as the value when a certain traveling wave problem has zero speed, completing the proof of Theorem~\ref{theo1}. Section~\ref{sec:application} puts in good use this theory on a simple model of neuronal differentiation. 

\section{Analysis of the parabolic problem}
\label{app:A}

We start with the parabolic problem associated with~\eqref{eq:GeneralElliptic}
\begin{equation}\label{eq:GeneralParabolic}
	\begin{cases}
	 \partial_t A-\varepsilon d_A\Delta A = A \, H_A(x,A,B),  \qquad 0 < x < 1, \; t \geq 0,
	 \\[5pt]
	 \partial_t B-\varepsilon d_B\Delta B = B\, H_B(x,A,B),
	\end{cases}
\end{equation} 
completed again with the Robin boundary conditions~\eqref{eq:BoundaryConditions}.

We show that for a well chosen pair of initial conditions, solutions to the parabolic problem~\eqref{eq:GeneralParabolic}-\eqref{eq:BoundaryConditions} are monotonic in time. Since all coefficients are regular, solutions are classical and therefore bounded. From here, the existence of steady states is granted. 

Because $F_B$ is an increasing function in $(x_b,1]$ we can expect that any non-negative solution for the second equation of~\eqref{eq:GeneralParabolic} is upper bounded by $F_B(1)$. Under the change of variables $\mathcal B=F_B(1)-B$, system~\eqref{eq:GeneralParabolic} becomes
\begin{equation}\label{eq:invB}
\begin{cases}
 \partial_t A-\varepsilon d_A\Delta A = AH_A\big(x,A,F_B(1)-\mathcal B\big),  \qquad 0 < x < 1, \; t \geq 0,
 \\[5pt]
 \partial_t \mathcal B-\varepsilon d_B\Delta \mathcal B = -(F_B(1)-\mathcal B)H_B\big(x,A,F_B(1)-\mathcal B\big),
\end{cases}
\end{equation}
with the respective boundary conditions
\begin{equation}\label{eq:BoundConditions2}
\begin{cases}
A(0)-\sqrt{\varepsilon}\,\frac{\partial}{\partial x}A(0)=F_A(0),\qquad A(1)+\sqrt{\varepsilon}\,\frac{\partial}{\partial x}A(1)=0,
\\[5pt]
\mathcal B(0)-\sqrt{\varepsilon}\,\frac{\partial}{\partial x}\mathcal B(0)=F_B(1),\qquad \mathcal B(1)+\sqrt{\varepsilon}\,\frac{\partial}{\partial x}\mathcal B(1)=0.
\end{cases}
\end{equation}
hypothesis~\ref{GH2}, \ref{GH1a} and \ref{GH1b} imply that the pair $(0,0)$ (respectively $(F_A(0),F_B(1))$) is a sub-solution (resp. super-solution) of the steady state problem related to~\eqref{eq:invB}-\eqref{eq:BoundConditions2}. Therefore, taking $(0,F_B(1))$ as initial condition in~\eqref{eq:GeneralParabolic} we have the existence of a regular solution $(A_\e(t,x),B_\e(t,x))$ such that:
\begin{equation} \label{uniform_bounds0}
0 \leq A_\e(t,x) \leq F_A(0)\quad \text{and} \quad 0\leq B_\e(t,x) \leq F_B(1),\qquad t\geq0,1\leq x\leq1 .
\end{equation}

\begin{lemme}\label{lemme0}
Then for all $t\geq0$ and $x\in [0,1]$, we have $\partial_t A_\e(t,x)\geq0$ and $\partial_t B_\e(t,x)\leq0$.
\end{lemme}
\begin{proof}
Defining $u:=\partial_t A_\varepsilon$ and $v:=\partial_t B_\varepsilon$, we have
\begin{eqnarray*}
 \partial_t u -d_A\varepsilon \Delta u &=& uH_A+A\;\partial_AH_Au+A\;\partial_BH_Av,\\
 \partial_t v -d_B\varepsilon \Delta v &=& vH_B+B\;\partial_AH_Bu+B\;\partial_BH_Bv,
\end{eqnarray*}
multiplying the first equation by $u_-:=\min\{0,u\}$, the second one by $v_+:=\max\{0,v\}$ and integrating over $[0,1]$ we get
\begin{eqnarray*}
 \frac12\frac d{dt}\int u_-^2 +d_A\varepsilon\int |\partial_xu_-|^2 - d_A\varepsilon u_- \;\partial_xu_-\big|^1_0
 &=&
 \int u_-^2\big(H_A+A\;\partial_AH_A\big)+\int A\;\partial_BH_Au_-v,\\
 \frac12\frac d{dt}\int v_+^2 +d_B\varepsilon\int |\partial_xv_+|^2 -d_B\varepsilon v_+ \;\partial_xv_+\big|^1_0&=& \int v_+^2\big(H_B+B\;\partial_BH_B\big)+\int B\;\partial_AH_Bv_+u.
\end{eqnarray*}
Time continuity of $(A_\e(t,x),B_\e(t,x))$ together with initial conditions imply that for any $x$: $$u_-(0,x)=0\quad\text{and}\quad v_+(0,x)=0.$$ Thus, there exists $C>0$ such that
\begin{equation}\nonumber
 \frac{d}{dt}\int(u_-^2+v_+^2)\leq C\int (u_-^2+v_+^2),
\end{equation}
with zero initial condition. We conclude using Gr\"onwall's lemma.
\end{proof}

\subsection{Monotonicity in space}

We have shown that the monotonicity property of the maps $H_A$ and $H_B$ in space implies monotonicity of $F_A(x)$ and $F_B(x)$, solutions of the zero diffusion problem at location $x$. This is also true of the maps $(A_\varepsilon,B_\varepsilon)$  solutions of the parabolic equation~\eqref{eq:GeneralParabolic}. In detail, we show that monotonic initial conditions ensure monotonic solutions $(A_\varepsilon,B_\varepsilon)$ in space for all times. This property has two remarkable implications: time dependent solutions belong to the bounded variation class and also their respective steady states.

\begin{lemme}\label{lemma2}
  For any $\varepsilon>0$ fixed, let us consider any solution $(A_\varepsilon,B_\varepsilon)$ of~\eqref{eq:GeneralParabolic}-\eqref{eq:BoundaryConditions} with initial conditions $A(0,x)$ decreasing and $B(0,x)$ increasing. Under assumption~\eqref{uniform_bounds0}, we have for all $t\geq0$ 
\[\frac{\partial}{\partial x} A_\varepsilon(t,x)\leq0\quad\text{and}\quad\frac{\partial}{\partial x} B_\e(t,x)\geq0,\quad 0\leq x\leq1.\]
\end{lemme}

\begin{proof}
The proof proceeds as that of Lemma~\ref{lemme0}: we define $u:=\partial_x A_\varepsilon$ and $v:=\partial_x B_\varepsilon$, we have
\begin{eqnarray*}
 \partial_t u -d_A\varepsilon \Delta u &=& uH_A+A\;\partial_xH_ A+A\;\partial_AH_Au+A\;\partial_BH_Av,\\
 \partial_t v -d_B\varepsilon \Delta v &=& vH_B+B\;\partial_xH_B+B\;\partial_AH_Bu+B\;\partial_BH_Bv,
\end{eqnarray*}
multiplying the first equation by $u_+$ and the second one by $v_-$, integrating over $[0,1]$ and using that boundary conditions~\eqref{eq:BoundaryConditions} and~\eqref{uniform_bounds0}, we get
\begin{eqnarray*}
 &&\frac12\frac d{dt}\int u_+^2 +d_A\varepsilon\int |\partial_x u_+|^2 - d_A\e u_+\; \partial_xu_+\big|^1_0
 =\\
 &&\qquad\qquad  \int u_+^2\big(H_A+A\;\partial_AH_A\big)+\int A\;\partial_BH_Au_+v+\int A\;\partial_xH_Au_+\leq C\int u_+^2-C\int u_+v_-,\\
&& \frac12\frac d{dt}\int v_-^2 +d_B\varepsilon\int |\partial_xv_-|^2 -d_B\e v_-\; \partial_xv_-\big|^1_0=\\
 &&\qquad\qquad \int v_-^2\big(H_B+B\;\partial_BH_B\big)+\int B\;\partial_AH_Bv_-u+\int B\;\partial_xH_B v_-\leq C\int v_-^2-C\int u_+v_-,
\end{eqnarray*}
where we have also used that $\partial_xH_A\leq0$ and $\partial_xH_B\geq0$. It is then easy to see that there exists $C>0$ such that:
\begin{equation}\nonumber
 \frac{d}{dt}\int(u_+^2+v_-^2)\leq C\int (u_+^2+v_-^2),
\end{equation}
with zero initial condition, and to conclude the proof using Gr\"onwall's lemma.
\end{proof}

We have therefore constructed a pair $(A_\e(t,x),B_\e(t,x))$ such that~\eqref{uniform_bounds0} is satisfied. We can apply Lemma~\ref{lemma2} and find that for any time $A_\e(t,x)$ is decreasing and $B_\e(t,x)$ increasing. Moreover, Lemma~\ref{lemme0} together with~\eqref{uniform_bounds0} imply that pointwise in space $A_\e(t,x)$ (resp. $B_\e(t,x)$) converges to $A_\e(x)$ (reps. $B_\e(x))$ solution to~\eqref{eq:GeneralElliptic}, together with the boundary condition~\eqref{eq:BoundaryConditions},  
in the weak sense. Bootstrap method allows us to conclude that $A_\e(x),B_\e(x)\in C^2(0,1)\cap C_0([0,1])$ which proves the first part of Theorem~\ref{theo1} and~\eqref{eq:monotone}.

\subsection{Positivity of the solutions}\label{app:B}
We now consider the pair $(A_\e,B_\e)$ solution of the stationary problem~\eqref{eq:GeneralParabolic}-\eqref{eq:BoundaryConditions} for $\e>0$ fixed. We now provide finer estimates of sub-solutions in order to control $A_\e(0)$ and $B_\e(1)$ away from zero.

\begin{prop}\label{prop2}
 There exists $\e_0>0$ such that for any $\e<\e_0$, $A_\e$ is strictly positive and $A_\e(0)$ is, uniformly in $\e$, larger  than some $\delta_A>0$. The same holds for $B_\e$ and $B_\e(1)$.
\end{prop}

\begin{proof}The proof consists in finding a strictly positive sub-solution for
\begin{equation}\label{auxiliar}
 -\varepsilon d_A \frac{d^2}{d x^2}\phi_A= H_A\big(x,\phi_A(x),B_\varepsilon(x)\big) \phi_A,
\end{equation}
i.e., the equation for $A_\e$ when $B_\e$ is fixed. To this purpose, we analyze a completely solvable linear problem related to~\eqref{auxiliar}, whose solution constitutes a sub-solution of~\eqref{auxiliar} and is defined and strictly positive up to the boundary. This solution can thus be used to find a lower bound for $A_\e(0)$.

Consider the following linear equation
\begin{equation}\label{ap:linearA}
 -\varepsilon d_A \frac{d^2}{d x^2}\phi_A=\bigg[\min_{0\leq s\leq F_A(0)}H_A\big(1,s,F_B(1)\big)\bigg] \phi_A
\end{equation}
with boundary conditions inherited from~\eqref{eq:BoundaryConditions}:
\begin{equation}\label{ap:BoundConditions3}
 \phi_A(0)-\sqrt{\varepsilon}\,\frac{d}{d x}\phi_A(0)=F_A(0), \qquad
  \phi_A(1)+ \sqrt{\varepsilon}\,\frac{d}{d x}\phi_A(1)=0.
\end{equation}
Clearly, the solution takes the form
\begin{equation}\nonumber
 \phi_A(x)=\alpha_\e\, e^{x\sqrt{\mu/\varepsilon}}+\beta_\e\, e^{-x\sqrt{\mu/\varepsilon}},\quad \text{with}\quad d_A\mu=-\big[\min_{0\leq s\leq F_A(0)}H_A\big(1,s,F_B(1)\big)\big]>0.
\end{equation}
Using~\eqref{ap:BoundConditions3}, one can find the exact values of $\alpha_\e$ and $\beta_\e$ as a function of the system parameters
\begin{equation}\nonumber
\frac{\alpha_\e}{\beta_\e}=\frac{(\sqrt{\mu}-1)}{(\sqrt{\mu}+1)}e^{-2\sqrt{\mu/\varepsilon}}\quad\text{and}\quad\beta_\e=\frac{F_A(0)}{\sqrt{\mu}+1}\Big[1-\frac{(\sqrt{\mu}-1)^2}{(\sqrt{\mu}+1)^2}e^{-2\sqrt{\mu/\varepsilon}}\Big]^{-1}.
\end{equation}

Taking $\e\rightarrow0$, we immediately compute
\begin{equation}\nonumber
\alpha_\e\, e^{\sqrt{\mu/\e}}\rightarrow0\quad\text{and}\quad \beta_\e\rightarrow\beta:=\frac{F_A(0)}{\sqrt{\mu}+1},
\end{equation}
thus, for any $\varepsilon>0$ small enough, $\phi_A$ becomes positive and
  \begin{equation}\nonumber
  0<\min_{0\leq x\leq1}\phi_A(x)\leq\max_{0\leq x\leq1}\phi_A(x)\leq |\alpha_\e|e^{\sqrt{\mu/\e}}+\beta_\e\leq F_A(0).
 \end{equation} 
Then, using that $H_A$ is decreasing in both $x$ and $B$, we obtain
\begin{equation}\nonumber
 -\varepsilon d_A \frac{d^2}{d x^2}\phi_A=\big[\min_{0\leq s\leq F_A(0)}H_A\big(1,s,F_B(1)\big)\big]\phi_A\leq H_A\big(x,\phi_A(x),B_\varepsilon(x)\big) \phi_A.
\end{equation}
Therefore, $\phi_A$ is a  sub-solution to~\eqref{auxiliar} comprised between $0$ and $F_A(0)$. Since $A_\e$ is a solution to the same problem with the same bounds and $\phi_A(0)$ is converging to $\beta$, the existence of $\delta_A>0$ follows.
\end{proof}

\section{Asymptotic analysis as $\varepsilon$ vanishes and front position}
\label{sec:vanishes}

We now consider the monotonic stationary solutions $(A_\varepsilon,B_\varepsilon)$ for $\e>0$ defined in Theorem~\ref{theo1}. Thanks to Proposition~\ref{prop2}, we know that for any $x\in[0,1]$
\begin{equation} \label{uniform_bounds}
0 < A_\varepsilon (x) \leq F_A(0)\quad \text{and} \quad 0 < B_\varepsilon (x) \leq F_B(1).
\end{equation}

We are now in a position to demonstrate the convergence of the pair $(A_\e, B_\e)$ as $\e\to 0$ towards a pair $(A_0,B_0)$ that are discontinuous at the same point $x^*$ and are characterized by point (i) of Theorem~\ref{theo1}. The proof proceeds as follows. First, using the monotonicity of $(A_\e,B_\e)$ we find the existence of $A_0$ and $B_0$, and we characterize those limits as a family of critical points of~\eqref{eq:GeneralElliptic} indexed by $x$. That characterization gives us three possibilities for the support of $A_0$. Using a WKB change of variables and the monotonicity properties of critical points (characterized by lemma~\ref{lem:Monotonicity}), we discard two of them. This allows to conclude on the existence of a unique $x^*\in[x_b,x_a]$ with the properties stated in Theorem~\ref{theo1}.
\subsection{The limit as $\e$ vanishes}

We recall that by monotonicity and $L^\infty$ bounds, the total variations of $ A_\varepsilon$ and $ B_\varepsilon$ are uniformly bounded in $\varepsilon$. Classical theory of Bounded Variation functions (see for instance~\cite[Theorem 4, p.176]{evans1991measure}) ensures that there exists a subsequence $\varepsilon_k$ and BV-functions $A_0$,  $B_0$ such that, almost everywhere and in all $L^p(0,1)$, $1\leq p < \infty$, 
\begin{equation}\label{eq:ExSolution}
 \begin{cases}
  A_{\varepsilon_k} \; \longrightarrow \; A_0,  \qquad 0 \leq A_0(x)  \leq F_A(0),  \qquad \frac{d}{dx} A_0 \leq 0, 
  \\[5pt]
  B_{\varepsilon_k} \; \longrightarrow \; B_0, \qquad  0  \leq B_0(x) \leq F_B(1), \qquad \frac{d}{dx} B_0 \geq 0,  .
 \end{cases}
\end{equation}
 Those limits satisfy, almost everywhere, 
 \begin{equation}\label{eq:ExCondition}
 \begin{cases}
  A_0 H_A\big(x,A_0(x),B_0(x)\big)=0,
  \\[5pt]
  B_0H_B\big(x,A_0(x),B_0(x)\big)=0.
 \end{cases}
\end{equation}

This means that at each point $x$, $( A_0, B_0)$ is one of the four nonnegative equilibrium points; $(0,0)$ and those three given by hypothesese~\eqref{GH1a}, \eqref{GH1c}. Because $A_0$ is decreasing, three possible scenarios arise:

\begin{itemize}
\item[(a)] There exits $x^*$ such that $(A_0(x), B_0(x)) = (F_A (x),0) $,  for $x< x^*$ and $ (A_0(x),B_0(x)) = (0,F_B (x)) $,  for $x> x^*$. 

\item[(b)] There exists $x_-^*< x_+^*$ such that  $ (A_0(x),B_0(x)) = (F_A (x),0) $,  for $x< x^*_-$, $ (A_0(x), B_0(x)) = (0,0) $,  for $x_-^*<x< x_+^*$.

\item[(c)]  There exists $x^*\geq x_b$ such that $ (A_0(x),B_0(x)) = (F_A (x),0) $,  for $x< x^*$   and $ (A_0(x),B_0(x)) =(A^*, B^*)$ for $x>x^*$ close enough to~$x^*$.
\end{itemize}

Notice that neither (a) nor (b) exclude the possibility that $A_0$ is identically zero. Indeed, at this stage, $x^*$ (or $x^*_-$) could be the origin. Our aim now is to show that only the first scenario is possible for some $x^*\in[x_b,x_a]$ proving part (i) of Theorem~\ref{theo1}.

Scenario (c) can be readily discarded. Indeed, if (c) holds, then the relationship~\eqref{inversed} would be in contradiction with the monotonicity of $A_0(x)$ in a neighborhood of $x^*$.


\subsection{WKB change of unknown}

We define $\varphi_A^\varepsilon:=-\sqrt{\varepsilon}\log(A_\varepsilon)$, which is well defined thanks to Proposition~\ref{prop2}. Furthermore,
\begin{equation}\nonumber
 \frac{d}{dx}\varphi_A^\varepsilon=-\sqrt{\varepsilon}\;\frac{\frac{d}{dx}A_\varepsilon}{A_\varepsilon}\qquad\text{and}\qquad\frac{d^2}{dx^2}\varphi_A^\varepsilon=-\sqrt{\varepsilon}\;\Big(\frac{\frac{d^2}{dx^2}A_\varepsilon}{A_\varepsilon}-\frac{|\frac{d}{dx}A_\varepsilon|^2}{A_\varepsilon^2}\Big),
\end{equation}
and we find that $\varphi_A^\varepsilon$ is solution of the eikonal equation
\begin{equation}\nonumber
 \Big|\frac{d}{dx}\varphi_A^\varepsilon\Big|^2-\sqrt{\varepsilon}\;\frac{d^2}{dx^2}\varphi_A^\varepsilon=-H_A(x,A_\varepsilon,B_\varepsilon),
\end{equation}
with
\begin{equation}\nonumber
 \frac{d}{dx}\varphi_A^\varepsilon(0)=\frac{F_A(0)}{A_\varepsilon(0)}-1, \qquad
 \frac{d}{dx}\varphi_A^\varepsilon(1)=1.
\end{equation}
The same constructions can be made for $\varphi_B^\varepsilon$. If we prove that the family $(\varphi_A^\e)$ has some regularity, then we can take let $\e$ go to 0 in $\varphi_A^\e$ and $\varphi_B^\e$. That is the object of the following:

\begin{lemme}\label{lemme3}
 The sequence $(\varphi_A^\varepsilon)$ is uniformly Lipschitz-continuous with respect to $\varepsilon$. Therefore, after extracting a subsequence, $\varphi_A^{\varepsilon_k}  \underset{ \varepsilon_k \to 0 }{\longrightarrow} \varphi_A^0$,  a Lipschitz continuous, non-decreasing  viscosity  solution of
 \begin{equation} \label{eq:viscA}
 \Big|\frac{d}{dx}\varphi_A^0\Big|^2=-H_A(x,A_0,B_0).
\end{equation}
 
The same construction for $B_\varepsilon$ provides us with a function $\varphi_B$, Lipschitz continuous, non-increasing viscosity solution of
\begin{equation} \label{eq:viscB}
 \Big|\frac{d}{dx}\varphi_B^0\Big|^2= -H_B(x,A_0,B_0).
\end{equation}
\end{lemme}

\begin{proof}

Since $A_\varepsilon\geq0$ and $\frac{d}{dx} A_\varepsilon \leq 0$ we get directly that $\frac{d}{dx}\varphi^\varepsilon_A\geq0$. We are going to prove that there exists $C_{\e_0}$, independent of $\varepsilon$, such that 
$$
0\leq\sup_{x\in[0,1]}\frac{d}{dx}\varphi_A^\varepsilon(x)\leq C_{\e_0}.
$$ 

Consider $y$, one argmax of $\frac{d}{dx}\varphi_A^\varepsilon(y)$. If $0<y<1$, then $\frac{d^2}{dx^2}\varphi_A^\varepsilon(x)=0$ and
 \begin{equation}\nonumber
  \Big|\frac{d}{dx}\varphi_A^\varepsilon\Big|^2=-H_A(x,A_\varepsilon,B_\varepsilon),
 \end{equation}
which is uniformly upper-bounded because $H_A$ is continuous and evaluated on $(0,1)\times(0,F_A(0))\times(0,F_B(1))$. The upper bound follows. 

If $y=0$, Proposition~\ref{prop2} tells us that $A_\varepsilon(0)$ is bounded from below by some positive constant $\delta_A$  independent from  $\e$. Then, we may conclude again because 
\begin{equation}\nonumber
 \frac{d}{dx}\varphi_A^\varepsilon(0)=\frac{F_A(0)}{A_\varepsilon(0)}-1\leq \frac{F_A(0)}{\delta_A}-1<\infty.
\end{equation}
If $y=1$, we immediately conclude thanks to the boundary condition and thus, we have proved the uniform Lipschitz estimate. 
\\

The Ascoli-Arzela theorem allows us to take a subsequence of $\varphi_A^\varepsilon$ which converges uniformly and we conclude thanks to the usual theory of viscosity solutions \cite{C.I.L:92, GB:94}. Note that the viscosity procedure only allows to control the limsup or liminf of the right hand sides of \eqref{eq:viscA}, \eqref{eq:viscB}, and this information sufficient for the conclusion we want to draw. 
\end{proof}

A direct consequence of Lemma~\ref{lemme3} is that scenario~(b) cannot hold. Indeed, in that case,
\begin{equation}\nonumber
 \Big|\frac{d}{dx}\varphi_A^0\Big|^2=-H_A(x,0,0)<0,\qquad\forall x\in(x_-^*,x_+^*)
\end{equation}
which is contradictory. 

The only possible scenario is therefore (a). In order to conclude the proof, we are left showing that $x^*\in[x_b,x_a]$. It suffices to show that $A_0(x)$ becomes positive when $x\rightarrow0$ and the same with $B_0(x)$ when $x\rightarrow1$.
\begin{lemme}\label{prop1}
 There exists two non empty intervals, namely $I_b$ and $I_a=[0,1]\setminus I_b$, such that $B_0\equiv0$ in $I_b$ and $A_0\equiv0$ in $I_a$. Moreover, 
 $$
 [0,x_b)\subset I_b\quad\text{ and }\quad(x_a,1]\subset I_a.
 $$
\end{lemme}

\begin{proof}
 Let us assume that there exists $y\in(x_a,1)$ such that $A_0(y)>0$.  We have shown that we are necessarily in scenario (a), which implies that $B_0(y)=0$ and by monotonicity 
 $$
 B_0(x)=0,\quad A_0(x)= F_A(x)\qquad \text{for }\;  0\leq x\leq y.
 $$

Using the fact that $(F_A(x),0)$ is linearly unstable for $x\in (x_a,1]$ and that $\varphi_B^0$ is a viscosity solution of~\eqref{eq:viscB}, we have
\begin{equation}\nonumber
 \Big|\frac{d}{dx}\varphi_B^0\Big|^2=-H_B(y,F_A(y),0)<0,
\end{equation}
 which is impossible, hence $A_0 \equiv 0$ on $(x_a,1)$. The same argument ensures us that $B_0\equiv0$ in $(0,x_b)$. One can therefore define the intervals $I_a$ and $I_b$ by maximality as the supports of $A_0$ and $B_0$.
\end{proof}

\section{Characterization of the Front}
\label{sec:characterization}

Now that we have proved the existence of a boundary $x^*$, we can turn to the characterization of this point. To this purpose, we start defining the point, $x^*_\e$ such that
$$
A_\e(x^*_\e) = B_\e (x^*_\e),
$$
which, by monotonicity, is unique. We also know, by compactness and unique limit, that  $x^*_\e  \rightarrow x^*$ when $\e\rightarrow0$.

We perform the change of variables $y=(x-x^*_\e)/\sqrt{\varepsilon},$ and define $a_\varepsilon(y)=A_\e (x^*_\e+\sqrt{\varepsilon}y)$ and $b_\varepsilon$ in the same way. System~\eqref{eq:Stationary} becomes
\begin{equation}\nonumber
\begin{cases}
-d_A\frac{d^2}{dy^2}a_\varepsilon(y) = a_\varepsilon(y)H_A\big(x^*_\e+\sqrt\varepsilon y,a_\varepsilon(y),b_\varepsilon(y)\big),
 \\[5pt]
-d_B\frac{d^2}{dy^2}b_\varepsilon(y) = b_\varepsilon(y)H_B\big(x^*_\e+\sqrt\varepsilon y,a_\e(y),b_\varepsilon(y)\big),
 \\[5pt]
 a_\varepsilon (0) = b_\varepsilon(0).
\end{cases}
\end{equation}

Because $a_\varepsilon$ and $b_\varepsilon$ are uniformly  bounded, by elliptic regularity they are uniformly bounded in  $C^2$ and, after extraction of a subsequence (by uniqueness, as we will show, in fact the full sequence converges), we may 
take the limit as $\e\to 0$ (which we know is well defined, bounded, Lipschitz-continuous).  We find that this limit, denoted $(a_0,\; b_0)$, is solution of
\begin{equation}\label{eq:Traveling0}
\begin{cases}
-d_A\frac{d^2}{dy^2}a_0(y) = a_0(y)H_A\big(x^*,a_0(y),b_0(y)\big),\qquad \partial_ya_0(y) \leq 0, 
 \\[5pt]
-d_B\frac{d^2}{dy^2}b_0(y) = b_0(y)H_B\big(x^*,a_0(y),b_0(y)\big), \qquad \partial_yb_0(y) \geq 0,
 \\[5pt]
 a_0(0) =b_0(0).
\end{cases}
\end{equation}

This solution is characterized as follows:
\begin{theo}\label{th:xstar} The limits satisfy $ a_0 \neq 0$, $b_0 \neq 0$  and there exists a unique value $x^*$ such that the system \eqref{eq:Traveling0} has a non-trivial solution. This solution is the unique traveling wave defined as 
\begin{equation}\label{eq:TravWave}
\begin{cases}
 -c(x)\frac{\partial}{\partial y}a(y;x)-d_A\frac{\partial^2}{\partial y^2}a(y;x) = a(y;x)H_A\big(x,a(y;x),b(y;x)\big), \qquad y \in \R, 
 \\[5pt]
 -c(x)\frac{\partial}{\partial y}b(y;x)-d_B\frac{\partial^2}{\partial y^2}b(y;x) = b(y;x)H_B\big(x,a(y;x),b(y;x)\big),
 \\[5pt]
  \displaystyle \lim_{y\rightarrow-\infty}a(y;x)=F_ A(x),\qquad\lim_{y\rightarrow+\infty}a(y;x)=0,
\\[5pt]
   \displaystyle  \lim_{y\rightarrow+\infty}b(y;x)=F_ B(x),\qquad\lim_{y\rightarrow-\infty}b(y;x)=0, 
\end{cases}
\end{equation}
with speed zero, that is $c(x^*)=0$,  and connecting $(F_A(x^*),0)$ to $(0,F_B(x^*))$.
\end{theo}

\begin{proof}
The proof is split into three steps. First we show that functions $a_\e$ and $b_\e$ cannot converge both at the same time to the zero function. Then, using that $a_0$ and $b_0$ converge at $-\infty$ to solutions of~\eqref{eq:ExCondition}, we show that limit conditions of~\eqref{eq:TravWave} are satisfied. Finally, thanks to a monotonicity argument on the speed $c(x)$, we show that $(a_0,b_0)$ are in fact the unique traveling wave solutions of~\eqref{eq:TravWave} such that $c(\cdot)=0$.

\noindent {\it 1st step.} 
The pair $(a_\varepsilon,b_\varepsilon)$ does not converge to the zero function.

Indeed,  for any interval $(y^-,0)$ with $y^- <0$, integrating by parts the equation on $a_\varepsilon$ after dividing it by $a_\varepsilon$, we compute 
 \begin{eqnarray*}
  \frac{1}{d_A}\int^{0}_{y^-}H_A\big(x^*_\e+\sqrt\varepsilon y,a_\varepsilon(y),b_\varepsilon(y)\big)\,dy
&=&-\int^{0}_{y^-} \frac{|\frac{d}{dy}a_\varepsilon|^2}{a_\varepsilon^2}\,dy-\Bigg[\frac{\frac{d}{dy}a_\varepsilon}{a_\varepsilon}\Bigg]^{0}_{y^-}\leq-\frac{\frac{d}{dy}a_\varepsilon(0)}{a_\varepsilon(0)}.
\end{eqnarray*}
Moreover, Lemma~\ref{lemme3} tells us that, for $\e<\e_0$,
$$
\frac{d}{dx}\varphi^{\e}_A = -\sqrt \e \;  \frac1{A_\e}\frac{d}{dx}A_\e<C_{\e_0}.
$$   
This implies directly that for any $y\in\R$
$$-\frac{\frac{d}{dy}a_\varepsilon(y)}{a_\varepsilon(y)}
=-\frac1{A_\e (x^*_\e+\sqrt{\varepsilon}y)}\frac{d}{dy}A_\e (x^*_\e+\sqrt{\varepsilon}y)
=-\frac{\sqrt\e\frac{dA_\e}{dx} (x^*_\e+\sqrt{\varepsilon}y)}{A_\e (x^*_\e+\sqrt{\varepsilon}y)}
\leq C_{\e_0} .
$$

Taking the limit $\e\rightarrow0$, using the continuity of $H_A$ and that $(a_\e,b_\e)\rightarrow(a_0,b_0)$ uniformly, we find 
\begin{equation} \label{eq:UniformIntegrability}
  \frac{1}{d_A}\int^{0}_{y^-}H_A\big(x^*,a_0(y),b_0(y)\big)\,dy\;\leq\;C_{\e_0} .
\end{equation}
If $(a_0,b_0) \equiv (0,0)$, then the left hand  side becomes $|y^-|H_ A(x^*,0,0)/d_A$ which goes to $\infty$ when $y^-\rightarrow-\infty$. Therefore, one of them, say $a_0$ is positive in some interval and by the strong maximum principle, $a_0(y) >0$ for any $y\in\R$. By the condition $a_0(0)=b_0(0)$, then $b_0$ is also positive. 
\\

\noindent {\it 2nd step.} 
The pair $(a_0,b_0)$ satisfies the conditions at infinity in~\eqref{eq:TravWave}.
  
We treat for instance the limit at $-\infty$. Again by elliptic regularity and thanks to \eqref{eq:UniformIntegrability}, $\frac{d^2}{d y^2}a_0(y)$ and $\frac{d^2}{d y^2}b_0(y)$ vanish at $- \infty$. Therefore the limits of $a_0$ and $b_0$ are steady state solutions with  $a_0(-\infty) > b_0(-\infty) $.

The case when this steady state is  $(A^*(x^*), B^*(x^*))$ is discarded by stability hypothesis~\eqref{GH1c} and saturation hypothesis~\eqref{GH2}. Indeed, we can rewrite the system defining $\epsilon_A= A(x^*)-a_0$ and $\epsilon_B=b_0-B(x^*)$. These functions are always positive and have non negative derivatives. Moreover, both they and their first derivatives, go to zero when $y\rightarrow-\infty$. We can write
\begin{equation}\nonumber
 \frac{d^2}{d y^2}\begin{pmatrix}\epsilon_A(y)\\ \epsilon_B(y)\end{pmatrix}\approx\begin{pmatrix} -\partial_AH_A/d_A & \partial_BH_A/d_A\\\partial_AH_B/d_B & -\partial_BH_B/d_B\end{pmatrix}\begin{pmatrix}\epsilon_A(y)\\ \epsilon_B(y)\end{pmatrix},
\end{equation}
where the matrix is evaluated at $(x^*,A^*(x^*),B^*(x^*))$ and we have neglected the terms of the type $\epsilon_A^2,\epsilon_B^2$ and $\epsilon_A\epsilon_B$ (which do not play a role in the analysis of the signs when $y\rightarrow-\infty$). Integrating between $-\infty$ and any value $y\ll -1$ we get
\begin{equation}\nonumber
 0\leq\frac{d}{d y}\begin{pmatrix}\epsilon_A(y)\\ \epsilon_B(y)\end{pmatrix}\approx\begin{pmatrix} -\partial_AH_A/d_A & \partial_BH_A/d_A\\\partial_AH_B/d_B & -\partial_BH_B/d_B\end{pmatrix}\begin{pmatrix}\int_{-\infty}^y\epsilon_A\\ \int_{-\infty}^y\epsilon_B\end{pmatrix},
\end{equation}
which is only possible when $\partial_AH_A\;\partial_BH_B-\partial_BH_A\;\partial_AH_B\geq0$ contradicting the saddle characterization of $(A^*(x),B^*(x))$.


\noindent {\it 3rd step.} Finally because the system is competitive, the positive solutions are unique and, in the case at hand, traveling waves with speed $0$.  We recall why the speed $c(\cdot)$ is monotonic. Considering the derivatives $w_a(y)=\frac{\partial}{\partial y}a(y)<0$, $w_b(y)=\frac{\partial}{\partial y}b(y)>0$ they satisfy
\begin{equation}\nonumber
 \begin{cases}
 -c(x)\frac{\partial}{\partial y}w_a(y;x)-d_A\frac{\partial^2}{\partial y^2}w_a(y;x) = M_{11}w_a+M_{12}w_b,
 \\[5pt]
 -c(x)\frac{\partial}{\partial y}w_b(y;x)-d_B\frac{\partial^2}{\partial y^2}w_b(y;x) = M_{21}w_a+M_{22}w_b.
 \end{cases}
\end{equation}
The signs $M_{12}:=\partial_BH_A<0$ and $M_{21}:=\partial_A H_B<0$ are compatible with the Krein-Rutman theory, and by consequence the dual problem has a signed solution
\begin{equation}\nonumber
 \begin{cases}
 c(x)\frac{\partial}{\partial y}\Phi_a(y;x)-d_A\frac{\partial^2}{\partial y^2}\Phi_a(y;x) = M_{11}\Phi_a+M_{21}\Phi_b,\qquad \Phi_a>0
 \\[5pt]
 c(x)\frac{\partial}{\partial y}\Phi_b(y;x)-d_B\frac{\partial^2}{\partial y^2}\Phi_b(y;x) = M_{12}\Phi_a+M_{22}\Phi_b,\qquad \Phi_b<0.
 \end{cases}
\end{equation}

We now consider the $x-$derivative: $z_a(y)=\frac{\partial }{\partial x} w_a(y;x)$ and $z_b(y)=\frac{\partial}{\partial x} w_b(y;x)$ satisfying
\begin{equation}\nonumber
 \begin{cases}
 -c(x)\frac{\partial}{\partial y}z_a(y;x)-d_A\frac{\partial^2}{\partial y^2}z_a(y;x) = M_{11}z_a+M_{12}z_b+c^\prime(x)w_a+a\;\partial_xH_ A,
 \\[5pt]
 -c(x)\frac{\partial}{\partial y}z_b(y;x)-d_B\frac{\partial^2}{\partial y^2}z_b(y;x) = M_{21}z_a+M_{22}z_b+c^\prime(x)w_b+b\;\partial_xH_B.
 \end{cases}
\end{equation}

Integrate in $y$ against the test function $\Phi$ and add the two lines, it remains
\begin{equation}\nonumber
 0=c^\prime(x)\int\underbrace{[w_a\Phi_a+w_b\Phi_b]}_{<0}\,dy+\int \overbrace{[a\Phi_a\;\partial_xH_ A+b\Phi_b\;\partial_xH_B]}^{<0}\,dy,\qquad 0<x<1,
\end{equation}
thus $c^\prime<0$. The uniqueness of $x^*$ follows directly.
\end{proof}

This result concludes the proof of theorem~\ref{theo1}. We now use this result on a simple model of differentiating neuronal tissue. 

\section{Application}\label{sec:application}

\subsection{Model}\label{sec:model}
As discussed in~\cite{quininao-prochiantz}, a classical illustration of neurodevelopment is provided by the compartmentalization of the neural tube in response to the diffusion of the ventral and dorsal morphogens Sonic Hedgehog (SHH) and Bone Morphogenetic Protein (BMP), respectively~\cite{ribes-etal:10,ulloa2007morphogens}. In this system, a continuous gradient activates ventral and dorsal genes, transcription factors are reciprocal inhibitor and self-activitor and diffuse through boundaries. This is well-known to result in the clear definition of territories that express distinct transcription factors subsets~\cite{ashe-briscoe:06,dessaud-etal:10,dessaud-etal:07,kiecker-lumsden:05}. 

We analyze a simplified version of the model proposed in~\cite{quininao-prochiantz}, which includes:
\begin{itemize}
	\item Epigenetic phenomena: the more a specie has been expressed, the more it is likely to be expressed. This phenomenon scales the production rate with a coefficient $\alpha_i(A,B)$. 
	\item The presence of morphogens with a graded concentration along the neural tissue, $F_i(x)$, $i\in \{A,B\}$,
	\item The self-activation of transcription factors
	\item and the saturation effects, limiting the production rate of each species proportionally to the total concentration within a cell.
	\item Eventually, diffusion of homeoproteins will be considered, through a small diffusion parameter $\e \ll 1$. 
\end{itemize}
We will show that these four mechanisms regulating the gene expression (response to gradients, self-activation, reciprocal inhibition and saturation) precisely correspond to our theoretical assumptions. Assuming that the number of cells is large, we consider a space-continuous description of the system, and we denote by $A(x)$ and $B(x)$ the concentrations of transcription factors at location $x$ on the neural tissue. The system described above readily translates into the system of parabolic equations:
\[
 \partial_t A -\varepsilon d_A\Delta A = \alpha_A A\big( F_A(x) + A) -\beta_A A (A+B), \qquad 0 < x <1, 
\]
and a 
similar equation for $B$. In this equation, we considered epigenetic phenomena to have linear effects: $\alpha_A(A,B)=\alpha_A A$. Therefore, the term $\alpha_A A$ is the transcriptional intensity, $\beta_A$ is the saturation parameter, and we assume $0<\alpha_A<\beta_A$ because saturation will overcome necessarily the self-activation. The parameter $d_A$ incorporates the relative level of diffusion of the parameter $A$ compared to that of $B$ (at least one of these constants can be incorporated in the $\e$). We shall assume that the system is subject to Robin type boundary conditions~\eqref{eq:BoundaryConditions}. 

It is not hard to rescale the system so as to write the stationary solutions in the form:
\begin{equation}\label{eq:Stationary}
\begin{cases}
 -\varepsilon d_A\Delta A = A\big(F_A(x)-A-s_{A}B\big), \qquad 0<x<1, 
 \\
 -\varepsilon d_B\Delta B = B\big(F_B(x)-B-s_{B}A\big),
\end{cases}
\end{equation}
where, for simplicity of notation, we use the same terms $F_ A (x)$ and $F_ B(x)$ to represent the rescaled action of external morphogen gradients. We introduce the parameters $s_i$ as positive constants taking into account the relation between $\alpha_i$ and $\beta_i$: $$s_A=\frac{\beta_A}{\beta_A-\alpha_A}>1\qquad\text{and}\qquad s_B=\frac{\beta_B}{\beta_B-\alpha_B}>1.$$

In the limit $\varepsilon$ goes to $0$, we look for a decreasing solution $A$ connecting the value $F_A(0)$ with 0.
The morphogen gradients are monotonic and smooth, assumed to be twice differentiable, defined on the closure of the domain, strictly positive and monotonic. Summarizing, there exists $\delta>0$ such that for any $0\leq x\leq 1$
\begin{equation}\label{H1}
 F_ A(x)>\delta,\quad\frac{d}{dx}F_ A(x)<0,\quad F_ B(x)>\delta,\quad\frac{d}{dx}F_ B(x)>0.
\end{equation}

We have mentioned that diffusion is extremely small. Non-trivial differentiation at these levels of diffusion would require that steady states for $\varepsilon=0$ are non-trivial as well. This is why we shall assume that:
\begin{equation}\label{H2}
\exists \; (x_a, \; x_b)\in I, \quad  \, x_b<x_a \quad \text{such that} \quad  F_A(x_b) = s_{A}F_ B(x_b),\quad F_ B(x_a) =s_{B} F_ A(x_a).
\end{equation}
A first remark is that combining assumptions~\eqref{H1} and~\eqref{H2} we get that
 \begin{equation}\label{eq:rmq1}
 \begin{cases}
 F_B(x)< F_ B(x_a) = s_{B}F_ A(x_a)< s_{B}F_ A(x),\quad \text{for } x\in [0,x_a),
 \\
 F_ A(x)< F_A(x_b) = s_{A}F_ B(x_b)< s_{A}F_ B(x),\quad \text{for } x\in (x_b,1].
\end{cases}
\end{equation}

We have already noticed that both saturation coefficients $s_{A}$ and $s_{B}$ are greater than 1. For the sake of generality, we make the weaker assumption: 
\begin{equation}\label{H3}
s_{A}s_{B} > 1. 
\end{equation}

Of course, in these notations, the parabolic system reads:
\begin{equation}\label{eq:Neurodevelopment}
\begin{cases}
 \partial_t A-\varepsilon d_A\Delta A = A\big(F_A-A-s_{A}B\big),  \qquad 0 < x < 1, \; t \geq 0,
 \\[5pt]
 \partial_t B-\varepsilon d_B\Delta B = B\big(F_B-B-s_{B}A\big),
\end{cases}
\end{equation}
with the Robin boundary conditions~\eqref{eq:BoundaryConditions}. If~\eqref{GH2}--\eqref{GH1c} are met for these $H_A$ and $H_B$, then Theorem~\ref{theo1} allow us to say that starting with monotonic initial conditions, then solution to~\eqref{eq:Neurodevelopment} defines a unique point $x^*$ as a boundary between the two functional areas considered, disambiguating the boundary location. 

To start with,  note that assumption~\eqref{GH2} is valid thanks to \eqref{H1} and that they fit the  
interpretation for neurodevelopment. They are trivially checked in our case since the maps $H_A(x,A,B)$ and $H_B(x,A,B)$ are linear. We are therefore left to characterizing the equilibria of the system and their stability.

\begin{lemme} \label{lm:equi}
	The properties~\eqref{GH1a}--\eqref{GH1c} are valid for our model (see Fig.~\ref{fig:secondCase}). In details, under assumptions~\eqref{H1},~\eqref{H2} and~\eqref{H3} and in the absence of diffusion, we have 
	\renewcommand{\theenumi}{\roman{enumi}}
	\begin{enumerate}
		\item $(F_ A(x),0)$ is a stable fixed point for $x \in [0,x_a)$,
		\item $(0,F_ B(x))$ is a stable fixed point for $x \in (x_b,1]$,
		\item and there exists an additional solution, which is saddle, in $(x_a,x_b)$.
	\end{enumerate}
\end{lemme}

\begin{proof}
	First two fixed points are trivial solutions, and their stability is obtained by investigating the eigenvalues of the Jacobian matrix at these points
	\renewcommand{\theenumi}{\roman{enumi}}
	
	\begin{enumerate}
	\item At $\big(F_A(x),0\big)$, the Jacobian matrix reads
	 \begin{equation}\nonumber
	 \begin{bmatrix}
	 -F_A(x) & -s_{A}F_A(x)\\0 & F_B(x)-s_{B}F_A(x)
	 \end{bmatrix},
	 \end{equation}
	and~\eqref{eq:rmq1} ensures us that this point is stable only on the region $[0,x_a)$.
	   \item The pair $\big(0,F_ B(x)\big)$ which is analogous to the previous point 
	 and stable on $(x_b,1]$.
	\item  Because of hypothesis~\eqref{H3}, there is an extra fixed point $(A^*,B^*)$ given by
	 \begin{equation}\nonumber
	  A^*=\frac{ s_{A} F_ B - F_ A}{s_{A}s_{B}-1}, \qquad \text{ } \qquad 
	  B^*=\frac{ s_{B} F_ A - F_ B}{s_{A}s_{B}-1}.
	 \end{equation}
	 From~\eqref{eq:rmq1} and~\eqref{H3},  we get that $(A^*,B^*)$ is admissible (i.e. both coordinates are non-negative) only in the region $[x_b,x_a]$. Monotonicity properties are trivial from the explicit expression, and the stability is governed by the eigenvalues of the Jacobian matrix
	 \begin{equation}\nonumber
	 \text{Jac}(A^*,B^*)=-\begin{bmatrix}
	  A^*& s_{A}A^* \\
	  s_{B}B^* & B^*
	 \end{bmatrix},
	 \end{equation}
	 which has negative determinant (as a consequence of assumption~\eqref{H3}). Therefore, its eigenvalues are real with opposite signs, i.e. the point $(A^*,B^*)$ is a saddle fixed point, completing the proof. 
	 \end{enumerate}	
\end{proof}
\begin{remark}
	Let us eventually notice the following fact explaining the topology of the phase plane for $x\in(x_b,x_a)$. The space $\R_+\times \R_+$ is partitioned into the attraction basin of $(F_A(x),0)$ and that of $(0,F_B(x))$, in addition to lower-dimensional invariant manifolds. The attraction basins of the fixed point are separated by the one-dimensional stable manifold of the saddle fixed point $(A^*(x),B^*(x))$, which is an invariance manifold serving as a separatrix between those trajectories converging to $\big(F_ A,0\big)$ and $\big(0, F_B\big)$.
\end{remark}
\begin{figure}
\begin{minipage}{0.45\textwidth}
 \subfigure[The morphogen gradients and steady states ]{\includegraphics[width=\textwidth]{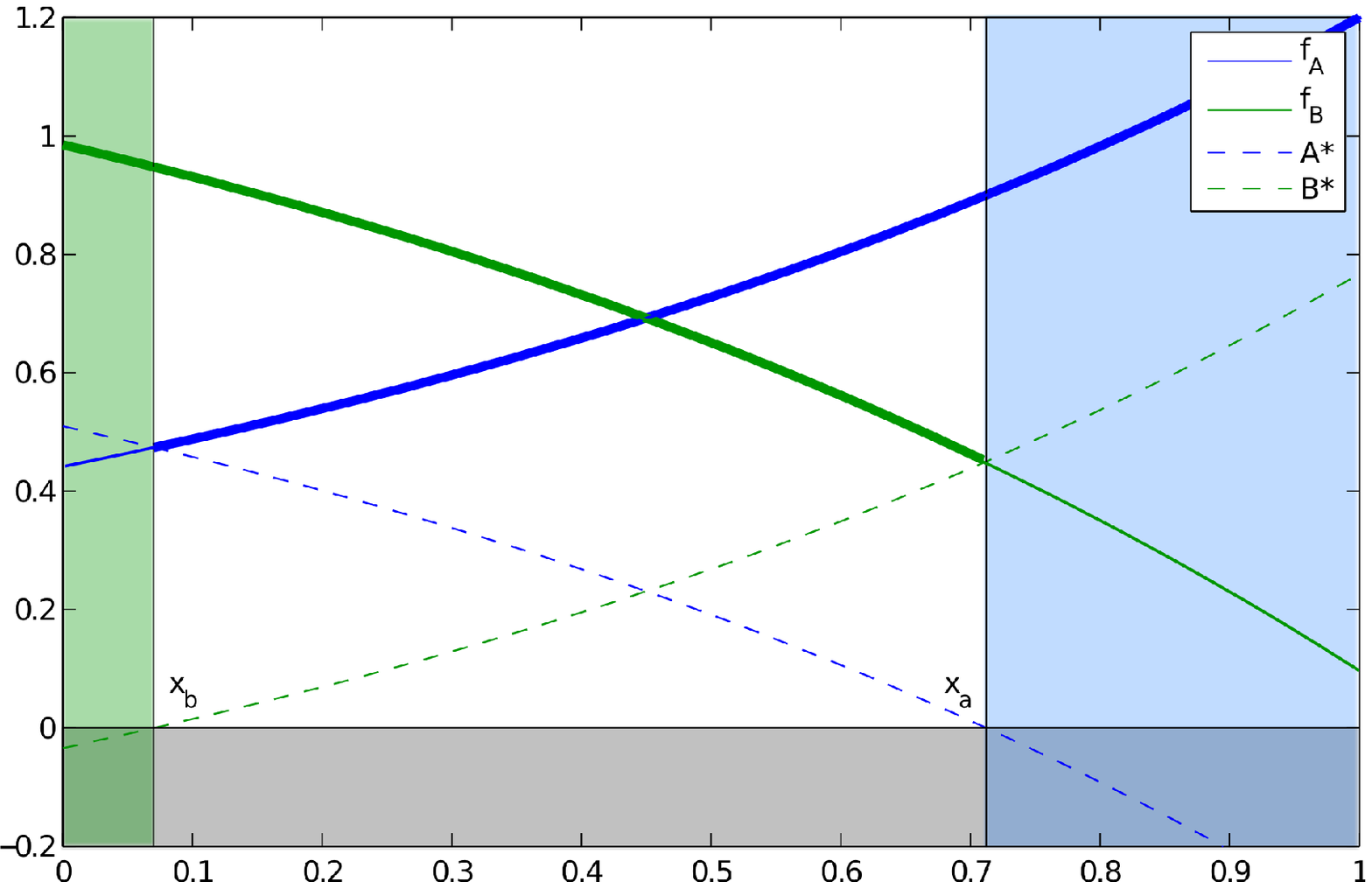}}
 \end{minipage}
 \begin{minipage}{0.25\textwidth}
  \subfigure{\includegraphics[width=\textwidth]{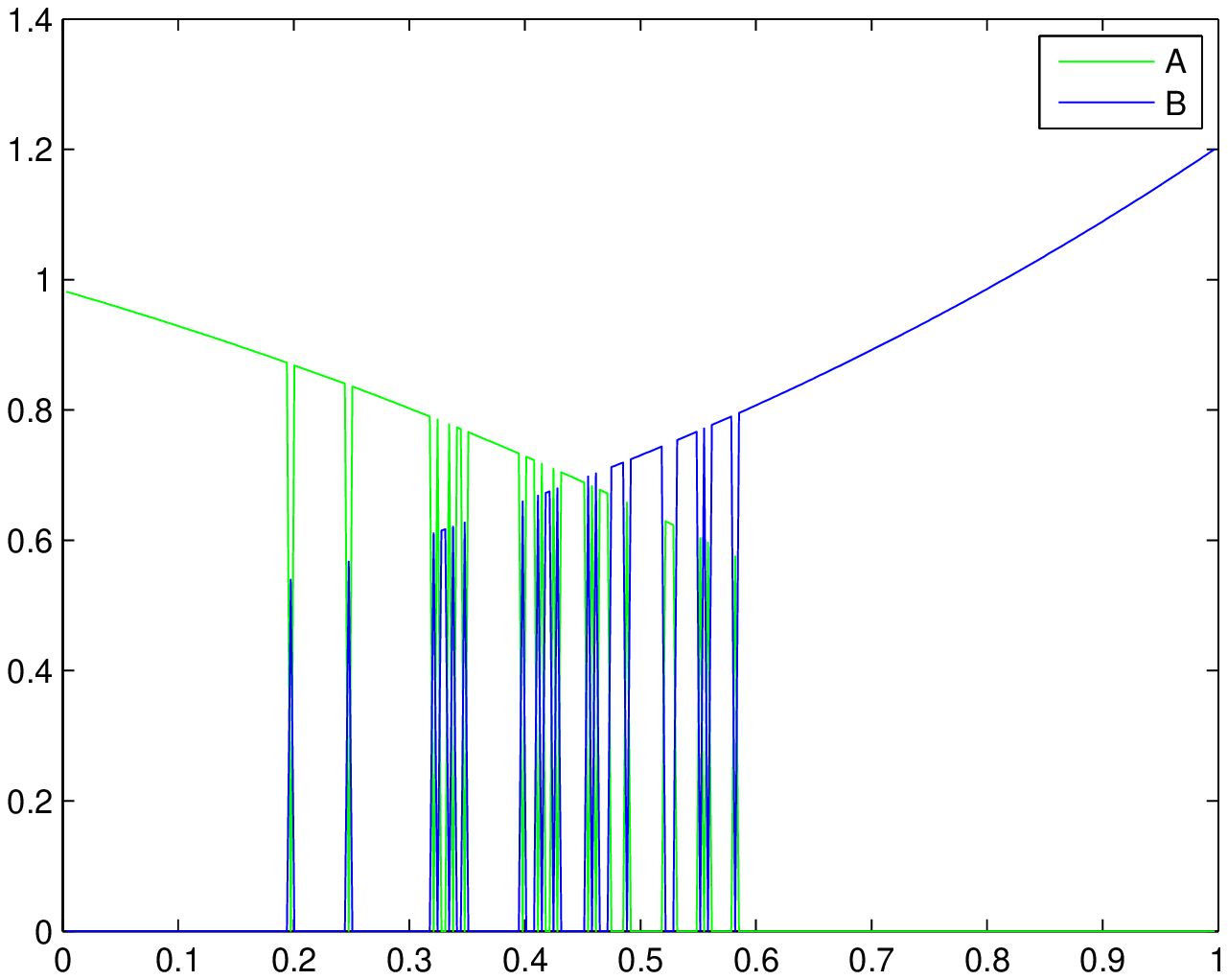}}
\\
 \subfigure{\includegraphics[width=\textwidth]{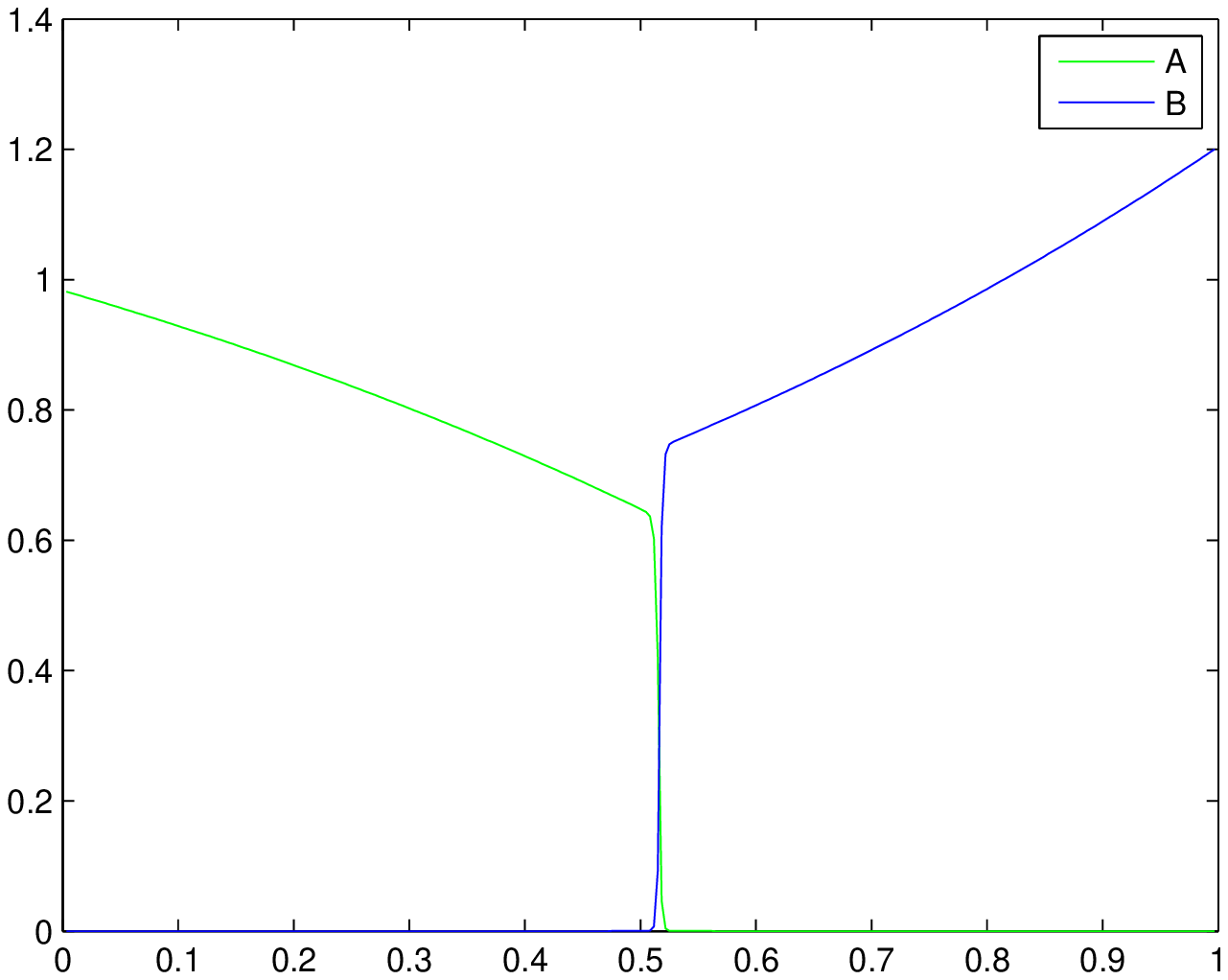}}
 \end{minipage}
 \vspace{-.5cm}
\caption{{Morphogenesis model with exponential morphogen gradients $F_A(x)$ and $F_B(x)$ (not shown) and $s_{A}=s_{B}=2$. (Left) Equilibria of the system in the absence of diffusion together with their stability (thick solid line: stable, thin solid line: repulsive, dashed: saddle). (Right) Numerical simulations of stationary states of the system~\eqref{eq:Neurodevelopment}-\eqref{eq:BoundaryConditions} shows (top) the ambiguity of boundary location for $\varepsilon=0$ and (bottom) the disambiguation for small diffusion $\varepsilon=10^{-6}$.}}
 \label{fig:secondCase}
\end{figure}

By a direct application of Theorem~\ref{theo1}, the system has a unique differentiated solution in the limit of small diffusion. But when considering only cell-autonomous mechanisms, the bistable region $x\in(x_b,x_a)$ induces an indeterminacy in the differentiation between two domains: cells may choose independently to differentiate into type $A$ or type $B$, yielding irregular and non-reproducible boundaries depending on the initial condition. This phenomenon is  illustrated in Figure~\ref{fig:secondCase}, right panel: in the absence of diffusion, the region within the interval $(x_b,x_a)$ has an unpredictable behavior that depends on space, while in the presence of even a very small diffusion, ambiguity disappears and a unique steady state emerges (see Figure~\ref{fig:secondCase}). In that sense, a small diffusion suffices to stabilize the transition. From an evolutionary viewpoint, endowing developmental transcription factors with non diffusion properties is a simple mechanism ensuring dramatic stabilization and robustness of the differentiation process. These numerical simulations further open some new perspectives. Indeed, we observe that the convergence towards the monotonic differentiated solutions seem to occur even when we relax the initial condition monotonicity hypothesis of Theorem~\ref{theo1}. Moreover, with random initial conditions, we numerically observe that for small times, $A$ converges rapidly to $F_A$ in $[0,x_b)$ and $B$ to $F_B$ in $(x_a,1]$, before the appearance of two abutting traveling fronts that develop toward the center of the coexistence zone, whose speed decreases as the solution converge. Proving that the theorem persists for general initial conditions remains an open problem.

\bibliographystyle{plain}

\end{document}